\documentclass[a4paper]{article}

\usepackage[english]{babel}
\usepackage[utf8]{inputenc}
\usepackage{amsmath,amsthm}
\usepackage{graphicx}
\usepackage[colorinlistoftodos]{todonotes}
\usepackage{enumitem}
\usepackage{amsfonts}
\usepackage{verbatim}
\usepackage{listings}
\usepackage[margin=1.0in]{geometry}
\usepackage{hyperref}
\usepackage{epstopdf}
\usepackage{comment}
\usepackage{color}
\newtheorem{thm}{Theorem}
\newtheorem{lem}{Lemma}
\newtheorem{prop}{Proposition}
\newtheorem{cor}{Corollary}
\newtheorem{exm}{Example}
\newtheorem{rem}{Remark}
\newtheorem{conj}{Conjecture}
\newtheorem{defn}{Definition}

\newcommand{\ssp}{{\mathcal{C}}}

\newcommand{\sspa}{{\mathcal{C}}_{\alpha}}

\title{Positivity for convective semi-discretizations}

\author{Imre Fekete, David I. Ketcheson, Lajos L\'oczi\footnote{Authors listed alphabetically.  
Emails: {\texttt{\{imre.fekete, david.ketcheson, lajos.loczi\}@kaust.edu.sa}}.
This work was supported by the King Abdullah University of Science and Technology (KAUST), 4700 Thuwal, 23955-6900, Saudi Arabia. The first author was also supported by the Tempus Public Foundation. The third author was also supported by the 
Department of Numerical Analysis, E\"otv\"os Lor\'and University, and the 
Department of Differential Equations, Budapest University of Technology and Economics, Hungary.}}

\date{\today}

\begin{document}
\maketitle

\begin{abstract}
We propose a technique for investigating stability properties like positivity
and forward invariance of an interval for method-of-lines discretizations, and
apply the technique to study positivity preservation for a class of TVD
semi-discretizations of 1D scalar hyperbolic conservation laws.
This technique is a generalization of the approach
suggested in \cite{K10}.  We give more relaxed conditions on the time-step for positivity
preservation  for slope-limited semi-discretizations
integrated in time with explicit Runge--Kutta methods.
We show that the step-size restrictions derived are sharp in a certain
sense, and that many higher-order explicit Runge--Kutta methods,
including the classical 4th-order method and all non-confluent methods with a
negative Butcher coefficient, cannot generally maintain positivity for these
semi-discretizations under any positive step size.
We also apply the proposed technique to centered finite difference
discretizations of scalar hyperbolic and parabolic problems.
\end{abstract}

\section{Introduction}

A number of important PDE models have the property that they preserve
positivity of the initial data:
\begin{align} \label{PDE-positivity}
    U(x,t_0) & \ge 0 \implies U(x,t)\ge 0 & \forall \ \ t\ge t_0,
\end{align}
or (more strongly) that they preserve the interval containing the initial data:
\begin{align} \label{PDE-invariance}
    U_\text{min} \le U(x,t_0) & \le U_\text{max} \implies U_\text{min} \le U(x,t) \le U_\text{max} & \forall \ \ t\ge t_0.
\end{align}
Examples of such PDEs include scalar hyperbolic conservation laws in one
spatial dimension, as well as the heat equation and some of its
generalizations.
In this work we study a technique that was first used in \cite{K10, K15}
for determining whether a certain
class of numerical discretizations satisfies the discrete analog
of \eqref{PDE-positivity} or \eqref{PDE-invariance}.

We focus on the application of this technique to the initial-boundary-value 
problem given by the hyperbolic conservation law \eqref{eq:conslaw} below,
together with positive initial and boundary data.
A common approach to solving hyperbolic conservation laws numerically
is to discretize in space with a slope  or flux limiter, and in
time with an explicit Runge--Kutta (explicit RK, or ERK) method.  It is natural
to ask whether the positivity property is retained under this discretization.
This question is usually analyzed by using Harten's theorem
\cite[Lemma~2.2]{harten1983a} to
show positivity under explicit Euler integration, and then
applying a higher-order strong stability preserving (SSP) 
method in time \cite{David11}.  This can be thought of as a method-of-lines
positivity analysis, in which the spatial and temporal discretizations
are analyzed separately.  In the present work, we perform a direct
positivity analysis of fully discretized schemes, obtaining stronger
results than what can be achieved by considering only Harten's theorem and
SSP methods.  These results provide a theoretical basis for some empirical
observations in \cite{hundsdorfer1995positive,2005_ketcheson_robinson},
wherein various Runge--Kutta integrators preserved strong stability
properties under step sizes much larger than those suggested by the
existing theory.

Following the usual terminology, the term {\em positivity} in the context of
positivity preservation is always meant herein in the weak sense; i.e., it
means {\em non-negativity}.  Although we focus on positivity to simplify the
presentation, the conditions we derive are necessary and sufficient for forward
invariance of an interval; see Theorem \ref{thm:invariance}.

Our paper is organized as follows.  In Section \ref{sec:TVD}, we review a
widely-used class of total-variation-diminishing (TVD) semi-discretizations for
hyperbolic conservation laws \eqref{eq:conslaw}, resulting in a system of ODEs with a specific structure. 
In Section \ref{sec:necsuff}, we introduce the concept of (positivity) step-size coefficient, and present our main theoretical results, showing sufficient or necessary
conditions for positivity preservation for this system of ODEs when integrated with an explicit
Runge--Kutta method. In the rest of the paper we investigate the step-size coefficients in some classes of ERK methods.
In Section \ref{sect:ERK(2,2)}, we summarize the results of \cite{K10} on the 
optimal SSP coefficient and positivity
step-size coefficient for two-stage, second-order ERK methods.
In Section \ref{sect:ERK(3,3)}, we present our main computational results for three-stage, third-order ERK methods.
We determine the optimal step-size coefficient in each of the three subclasses: methods with optimal step-size coefficient are essentially characterized in the two-parameter subclass,
whereas optimal methods in the two, one-parameter subclasses are completely described.
We also discover that the unique three-stage, third-order ERK method having the minimum
truncation-error coefficient also belongs to the set of methods with maximum
step-size coefficient.
In the proofs we are to find certain maximal hypercubes over which some multivariable polynomials are 
simultaneously non-negative. Section \ref{sect:ERKhigher} extends our analysis to higher-order ERK methods with more stages, and shows many negative results. 
Throughout Sections \ref{sect:ERK(2,2)}-\ref{sect:ERKhigher}, the ERK methods have been applied to the upwind semi-discretization of the advection equation. In Section \ref{sec:discussion}, we illustrate  the applicability of the proposed ideas to the centered spatial discretization of the advection equation and
to the heat equation. 
Finally, in Appendix \ref{sec:appendix} we present some \textit{Mathematica} code to generate the 
multivariable polynomials corresponding to an ERK method, and test their non-negativity in hypercubes. 

\subsection{TVD semi-discretizations under Runge--Kutta integration}\label{sec:TVD}

Semi-discretizations with a slope limiter can be analyzed as follows
\cite[Chapter III, Section 1.1-1.3]{Hundsdorfer03}---below $U$ will denote the PDE solution, while 
the corresponding lower-case letter is used in semi-discretized ODEs.
For concreteness we first consider the advection equation
\begin{align} \label{eq:advection}
    U_t + a(t) U_x & = 0
\end{align}
with $a>0$; the following analysis can be extended to arbitrary variations in
$a$.  We discretize in space by
\begin{align} \label{eq:fluxdiff}
  u'_k(t)=\frac{1}{\Delta x}\left(f_{k-\frac{1}{2}}-f_{k+\frac{1}{2}}\right),
\end{align}
where $\Delta x$ is the spatial mesh width, and the fluxes are given by
\begin{align*}
    f_{k+\frac{1}{2}} := a(t)\left(u_k+\psi(\theta_k)\left(u_{k+1}-u_k\right)\right).
\end{align*}
Here $\psi$ is known as the limiter function and $\theta_k$ is a ratio of
divided differences. This discretization leads to the form
\begin{align} \label{eq:advect-semi}
u_k'(t)=\frac{a(t)}{\Delta x}\left(1-\psi(\theta_{k-1})+\frac{1}{\theta_k}\psi(\theta_k)\right)(u_{k-1}(t)-u_k(t)).
\end{align}
More generally, we consider the scalar nonlinear conservation law
\begin{align} \label{eq:conslaw}
    U_t + f(U)_x & = 0,
\end{align}
where for simplicity we assume $f' \ge 0$.  The flux-differencing
semi-discretization \eqref{eq:fluxdiff} is again used, with the flux now
given by $f_{k+\frac{1}{2}}:=f(u_{k+\frac{1}{2}})$ where
\begin{equation}\label{eq:nonlincell}
u_{k+\frac{1}{2}}:= u_k+\psi(\theta_k)(u_{k+1}-u_k).
\end{equation}
By using the mean-value theorem, this can be written as
$$
u_k'(t)=\frac{f'(\eta_k)}{\Delta x}\left(1-\psi(\theta_{k-1})+\frac{1}{\theta_k}\psi(\theta_k)\right)(u_{k-1}(t)-u_{k}(t)),
$$
for some $\eta_k\in [u_{k-\frac{1}{2}},u_{k+\frac{1}{2}}]$.

It can be proved that non-negativity of \eqref{eq:fluxdiff} is ensured in either case if
\begin{align}\label{eq:Hundspositivity}
1-\psi(\theta)+\frac{1}{\tilde{\theta}}\psi(\tilde{\theta})\geq 0
\end{align}
for all $\theta\in\mathbb{R}$ and $\tilde{\theta}\in\mathbb{R}\setminus\{0\}$. 
A sufficient condition for \eqref{eq:Hundspositivity} is that 
\begin{align}\label{eq:Hundspositivity2}
0\leq \psi(\theta)\leq 1 \ \ (\theta\in\mathbb{R}),\quad \mathrm{and}\quad 0\leq\frac{1}{\theta}\psi(\theta)\leq\mu
\ \ (\theta\in\mathbb{R}\setminus\{0\})
\end{align}
for some value of $\mu$.  The upper bound $\mu$ will be important below when we
discretize in time.  Some well-known limiters 
satisfying conditions \eqref{eq:Hundspositivity2} include
\begin{itemize}
\item the minmod limiter \cite{R86}, with $\psi(\theta)=\max\left(0,\min\left(1,\theta\right)\right)$ and $\mu=1$;
\item the Koren limiter \cite{K93}, with $\psi(\theta)=\max\left(0,\min\left(1,\frac{1}{3}+\frac{1}{6}\theta,\theta\right)\right)$ and $\mu=1$;
\item the monotonized central limiter \cite{L77}, with $\psi(\theta)=\max\left(0,\min\left(2\theta,\frac{1+\theta}{2},2\right)\right)$ and $\mu=2$.
\end{itemize}

In this work we do not deal with the influence of boundary conditions, so for simplicity
we consider periodic boundaries.  The initial-boundary-value problem corresponding to \eqref{eq:fluxdiff} can thus be written,
with some suitable functions $q_k$, in the form
\begin{subequations}\label{eq:ODEs}
\begin{equation}
\begin{cases}
    u'_k(t)\  =   \displaystyle \ q_k(u(t),t)\ \frac{u_{k-1}(t)-u_k(t)}{\Delta x},&\quad k=1,\ldots,N, \\
    u_0(\cdot)\ :=u_N(\cdot), &  \\
    u_k(t_0) := u_k^0  & \quad k=1,\ldots,N.  \\
\end{cases}
\end{equation}
We assume that the initial conditions are non-negative
\begin{align}
    u_k^0   \ge \  0,   & \quad k=1,\ldots,N
\end{align}
and that \eqref{eq:Hundspositivity2} is satisfied, so that
\begin{align}
    0 \le q_k(v,t) \le \lambda_k(\mu+1)
\end{align}
\end{subequations}
holds for all $v, t$, with
$\lambda_k := \sup_{u\in[u_{k-\frac{1}{2}},u_{k+\frac{1}{2}}]} f'(u)$.
Under these assumptions, it can be shown that if the 
solution $u$ exists (which can be guaranteed, for instance, by assuming that
$q_k$ is Lipschitz), then it is positive \cite{K10}:
we have $u_k(t)\ge 0$ for $1\le k\le N$. 

Suppose now that we discretize \eqref{eq:ODEs} in
time with a given explicit Runge--Kutta method, and we want to preserve positivity. 
To achieve this, we aim to determine a value
$\Delta t_0>0$ such that the RK solution and stage values are non-negative
when applied to any problem of the form \eqref{eq:ODEs},
under some step-size restriction $\Delta t\leq \Delta t_0$. 

In the following examples we use the simplest RK method and give a suitable
value for $\Delta t_0$. 

\begin{exm}[Advection equation]\label{1stexample}
The forward Euler (FE) method preserves positivity for \eqref{eq:ODEs}
whenever
$$
\Delta t \leq \frac{\Delta x}{(\mu+1)\sup a(\cdot)} =: \tau_0.
$$
In other words, if $v$ denotes the vector $(v_1,\ldots,v_N)$ and we set $v_0:=v_N$, then
\[
v_k+\Delta t\cdot q_k(v,t)\ \frac{v_{k-1}-v_k}{\Delta x}\geq 0\ \text{for all}\ 0<\Delta t\leq \tau_0,\ v\geq 0
\textrm{ and } t\in [t_0,T],
\]
where $v\ge 0$ is understood componentwise.
\end{exm}

\begin{exm}[Scalar conservation law]\label{2ndtexample}
The FE method preserves positivity for \eqref{eq:ODEs} whenever
\begin{align*}
    \Delta t \leq \frac{\Delta x}{(\mu+1)\sup f'(U(\cdot,t_0))} =: \tau_0.
\end{align*}
\end{exm}

Given an ERK method, it is natural to consider the factorization 
\begin{equation}\label{factorization}
\Delta t\le \Delta t_0\equiv \gamma \tau_0
\end{equation} 
of the maximum allowed step size $\Delta t_0$, with $\tau_0$ appearing in Example \ref{1stexample} or 
\ref{2ndtexample}, 
because 
(i) this product structure naturally arises during the computations (see \eqref{xiproductform}), and (ii) 
the factor $\gamma$, referred to as the (positivity) step-size coefficient, will depend only on the chosen ERK method, whereas $\tau_0$ depends on the right-hand
side of the problem \eqref{eq:ODEs}. We will sometimes write $\gamma(A,b)$, where $(A,b)$ refer to the Butcher
coefficients of the RK method, to emphasize that the step-size coefficient
depends on the method coefficients. 
Our general goal in this paper is to find the maximum value of $\gamma\ge 0$.

From \eqref{factorization} it is also apparent that $\gamma$ plays a role similar to that of the SSP coefficient of the RK
method, denoted by $\ssp$, see \cite{David11}.  Specifically, an RK method is guaranteed to preserve positivity for
any positive system of ODEs with a general right-hand side $f(u(t),t)$, under the step-size restriction $\Delta t \le \ssp\Delta t_\textup{FE}$,
where $\Delta t_\textup{FE}$ is the positivity step-size threshold for the
forward Euler method. 
We will see that, for many RK methods, $\gamma$ is strictly 
larger than the SSP coefficient, because in determining a value $\gamma$ we
consider only problems of the form \eqref{eq:ODEs}. 

The present work is related to that presented in \cite{H13},
in that it proves less restrictive step sizes for positivity of a
particular class of problems.  However, the class of problems considered 
herein is different from the class of problems considered there.

\section{Positivity of the Runge--Kutta approximations}\label{sec:necsuff}
In this section we generalize and make systematic an approach that was
introduced in \cite{K10,K15}.  
Applying an explicit RK method to \eqref{eq:ODEs} we obtain
\begin{subequations} \label{eq:RK}
\begin{align}\label{eq:RK1}
y^i_k =     & u^n_k+\sum_{j=1}^{i-1}a_{ij}\xi_k^j\left(y^j_{k-1}-y^j_k\right),\quad k=1,\ldots,N,\ i=1,\ldots,m, \\
u^{n+1}_k = & u^n_k+\sum_{i=1}^{m}b_i   \xi_k^i\left(y^i_{k-1}-y^i_k\right),\quad k=1,\ldots,N, \label{eq:RK2}
\end{align}
\end{subequations}
where $u^n\approx u(t_n)$, $t_n=n\Delta t$ and 
\begin{equation}\label{xiproductform}
    \xi_k^j := \frac{\Delta t}{\Delta x} q_k(y^j,t_n+c_j\Delta t), \quad  j  = 1,\ldots,m.
\end{equation}
Here and throughout this work,
superscripts on $y$, $u$ and $\xi$ denote indices rather than exponents.
The coefficients $a_{ij}$ and $b_j$ can be conveniently represented by an
$m \times m$ strictly lower-triangular matrix $A$ and an $m\times 1$ vector $b$, respectively
 (referred to as the Butcher tableau $(A,b)$).

The idea is to express each component of the stages and of
the new solution as a combination of the previous solution values in the form
\begin{align}\label{eq:poly}
    u^{n+1}_k=\sum_{i=0}^{m} P_i(\xi) u^n_{k-i},
\end{align}
where the functions $P_i$ are multivariable polynomials depending only on the
method coefficients $(A,b)$, and $\xi$ is proportional to the time step size.
These polynomials are independent of $k$ since the method is translation-invariant
(though the values of $\xi$ do depend on $k$, in general).
We then aim to compute the largest step size that ensures non-negativity of these
polynomials.

Given a strictly lower-triangular matrix $A\in\mathbb{R}^{m\times m}$ and
vector $b\in\mathbb{R}^m$, it can be shown (see Section \ref{sec:polycomp}) that
the polynomials $P_i$ in \eqref{eq:poly} are multilinear functions of the
variables $\xi^j_\ell$, and there are $m+1$ polynomials and $m(m+1)/2$ variables.
Throughout this work we use the following ordering of the components of the vector $\xi$:
\[
\xi=(\xi^1_{k-(m-1)}, \xi^1_{k-(m-2)}, \ldots,\xi^1_{k-1}, \xi^1_{k};
\xi^2_{k-(m-2)}, \xi^2_{k-(m-3)}, \ldots, \xi^2_{k-1}, \xi^2_{k};\ldots; \xi^{m-1}_{k-1}, \xi^{m-1}_{k};\xi^m_{k}).
\]

\subsection{Computation of the multivariable polynomials}\label{sec:polycomp}

The polynomials $P_i$ can be obtained directly from \eqref{eq:RK}.
To anticipate their structure, we rewrite \eqref{eq:RK}
in matrix form.  We begin by rewriting \eqref{eq:RK1} as
\begin{align}\label{eq:Y}
	Y = e\otimes u^n + (A\otimes I_N) Q_{mN} (I_m\otimes D_N)Y,
\end{align}
where the vector $e:=(1,\ldots,1)\in\mathbb{R}^m$, $I$ is the identity matrix, the symbol $\otimes$ denotes the Kronecker product and subscripts denote the
dimensions of each matrix. 
The matrix $D_N$ is the $N\times N$ cyclic tridiagonal matrix with entries $1,-1,0$ such
that $[D_N]_{1,N}=1$, and $Q_{mN}$ is the $mN\times mN$ diagonal matrix with
entries $\xi^{1}_1,\xi^{1}_2,\dots,\xi^{1}_N, \dots, \dots ,\xi^{m}_N$. 
Let
$$M: = (A\otimes I_N) Q_{mN} (I_m\otimes D_N),$$
so that \eqref{eq:Y} becomes
\[
    Y=(I_{mN}-M)^{-1}(e \otimes u^n).
\]
Since $A$ is now strictly lower-triangular, the matrix  $M$ is strictly
block lower-triangular, and we can write
$$
Y = \sum_{i=0}^{\infty}M^i\left(e\otimes u^n\right) = \sum_{i=0}^{m-1}M^i \left(e\otimes u^n\right).
$$
Thus \eqref{eq:RK2} becomes
\begin{align}\label{eq:ERKstep-alt}
u^{n+1}=\left(I_N+(b^\top\otimes I_N) Q_{mN}(I_m\otimes D_N)\sum_{i=0}^{m-1}M^i(e\otimes I_N)\right)u^n.
\end{align}
\noindent\textbf{Remark.}\textit{
By using the relation
$$M^i = (A\otimes I_N) Q_{mN}\left((A\otimes D_N) Q_{mN}\right)^{i-1}  (I_m\otimes D_N),$$
we can also write \eqref{eq:ERKstep-alt} as
\[
u^{n+1} = \left( I_N + (b^\top\otimes I_N) Q_{mN} \sum_{i=0}^{m-1}\left( (A\otimes D_N) Q_{mN}\right)^{i}(e \otimes D_N) \right) u^n.
\]}

\subsection{A sufficient condition for non-negativity}\label{sect:suffcond}
 
In the present work we focus on positivity preservation, but---as we will see in Theorem \ref{thm:invariance} below---the step-size restrictions derived preserve a stronger property.  

First notice that, by consistency,
for any RK method we have
$$\sum_{i=0}^m P_i(\xi) = 1.$$
Thus if $P_i(\xi)\ge 0$, then \eqref{eq:poly} shows that each solution value is a convex combination of
solution values from the previous step. This motivates the following definition.

\begin{defn}\label{defn1}
Let an explicit Runge--Kutta method be given with coefficients $A, b$, and let
$P_i$ denote polynomials defined implicitly by \eqref{eq:poly} when the method
is applied to \eqref{eq:ODEs} and $\xi$ is defined by \eqref{xiproductform}.
The (positivity) step-size coefficient of the method is
\[
\gamma(A,b):=\sup\{\delta\ge 0 :  P_i(\xi)\ge 0 \textrm{ for each } 0\le i\le m \textrm{ and all } \xi\in[0,\delta]^{m(m+1)/2}\}. 
\]
We sometimes write simply $\gamma$, omitting the dependence on the method
coefficients when there is no potential for ambiguity.
If the set appearing in the $\sup$ above is empty, we set $\gamma:=0$. 
\end{defn}

\noindent Geometrically, the step-size coefficient is the edge length of the largest hypercube in the non-negative orthant over which the polynomials $P_i$ are all non-negative. 

From the above considerations it is clear that we have the following theorem.

\begin{thm} \label{thm:invariance}
Let an $m$-stage ERK method be given and let the polynomials $P_0,\dots,P_m$
be such that application of the method to \eqref{eq:ODEs} yields \eqref{eq:poly}.
Suppose that the time and space step sizes are chosen so that
\begin{equation}\label{timesteprestriction}
0\leq \Delta t\frac{q_k(u,t)}{\Delta x}\leq \gamma,\quad k=1,\ldots,N,
\end{equation}
for all values $u, t$.  
Let $u_\text{max}=\max_k(u_k(t_0))$, $u_\text{min}=\min_k(u_k(t_0))$.  
Then the solution given by the Runge--Kutta method applied
to \eqref{eq:ODEs} remains in the interval $[u_\text{min}, u_\text{max}]$.
\end{thm}

The invariance property appearing in Theorem \ref{thm:invariance}---that is, preservation of the interval containing the initial data---is the
discrete analog of a property known to hold for the exact
solution of \eqref{eq:conslaw}.  As a special case, when $u_\text{min}\ge 0$ we
have the property of positivity preservation.

\subsection{Necessary conditions for non-negativity}\label{sect:neccond}
Given a Runge--Kutta method $(A,b)$ and corresponding coefficient
$\gamma(A,b)$, we may ask whether taking the step size larger than
that permitted by \eqref{timesteprestriction} will always lead to negative
solution values.  The answer is ``no'', but we can construct particular
problems of the form \eqref{eq:ODEs} such that negative values will be
obtained.

The following theorems apply to non-confluent methods.
A Runge--Kutta method is said to be {\em non-confluent} if
the stage approximations all correspond to distinct times; i.e.,
if there are no distinct $i,j$ such that $c_i=c_j$.
\begin{thm}
    Let an explicit non-confluent Runge--Kutta method $(A,b)$ be given
    and let $\gamma(A,b)$ denote its step-size coefficient.
    Then for any $\tilde{\gamma}>\gamma$, there exists a problem
    \eqref{eq:ODEs}, a function $q = q(t)$ and a step size satisfying
    $$0 < \Delta t \le \tilde{\gamma} \frac{\Delta x}{\max_{k,t} q_k}$$
    such that the prescribed method
    leads to a negative solution at the first step.  
\end{thm}
\begin{proof}
    By our assumptions, there exist
    $\tilde{\xi} \in [0,\tilde{\gamma}]^{m(m+1)/2}$
    and an integer $j$ such that $P_j(\tilde{\xi}) < 0$.

    Take $N=m$ and
    $$u^0_k := u_k(t_0) = \begin{cases}  1 & \text{if } k=N-j \\
                                         0 & \text{otherwise.} \end{cases}$$
    Set $\Delta x =1$, $\Delta t = 1$, and 
    $q(\cdot,t_0+c_j\Delta t) = \tilde{\xi}^j$ for each $j$.
    Then direct computation reveals that
    $u_N^1 = P_j(\tilde{\xi})u_{N-j}^0 = P_j(\tilde{\xi}) < 0.$
\end{proof}

\begin{cor}
    Let an explicit non-confluent Runge--Kutta method be given.
    Then the time-step restriction \eqref{timesteprestriction}
    is sharp; i.e., it is the largest step size that guarantees
    positivity for all problems of the form \eqref{eq:ODEs}.
\end{cor}

The next result involves the concept of DJ-reducibility; see e.g.~\cite{dahlquist2006} for a definition. 
 Note that any reducible method is
equivalent to a method with fewer stages, so the irreducibility assumption here
is no essential restriction.

\begin{thm}\label{nonconfluenttheorem}
        Let an explicit non-confluent DJ-irreducible Runge--Kutta method
        $(A,b)$ be given such that at least one entry in the matrix $A$ or
        vector $b$ is negative. Then there exist
    initial data and a choice of $q$ such that the numerical 
    solution of \eqref{eq:ODEs} obtained with the method includes a negative
    value.  Therefore $\gamma(A,b)=0$.  
\end{thm}
\begin{proof}
Let a RK method be given as in the theorem.
We take initial data vector $v$ with $v_{p-1}=1$ where $p$ is an arbitrary 
grid index and
$v_k=0$ for all $k\ne p-1$.  
To simplify the presentation, in this proof only we define
$a_{m+1,j}:= b_j$ and $y_{m+1}:= u^{n+1}$.
Let $i$ be the index of the first RK stage with a negative
coefficient and suppose that $a_{iJ}<0$.  We have
\begin{align*}
	y^i_p & = v_p + \sum_{j=1}^m a_{ij} q_p(y^j,t_0+c_j \Delta t) \frac{\Delta t}{\Delta x} (y^j_{p-1} - y^j_p).
\end{align*}
We let
\begin{align*}
	q_p(u,t) = \begin{cases} 0 & t \ne t_0 + c_J \Delta t \\
    					   1 & t = t_0 + c_J \Delta t,
    \end{cases}
\end{align*}
so that
\begin{align*}
	y^i_p & = v_p + a_{iJ}  \frac{\Delta t}{\Delta x} (y^J_{p-1} - y^J_p).
\end{align*}
Furthermore, for all $j\le J$ we have $y^j=v$; in particular,
$y^J=v$.  Thus
\begin{align*}
	y^i_p & = v_p + a_{iJ}  \frac{\Delta t}{\Delta x} (v_{p-1} - v_p) = a_{iJ} \frac{\Delta t}{\Delta x} < 0.
\end{align*}
If $i=m+1$ (i.e., if there is some $i$ such that $b_i<0$), then we are done.

Suppose to the contrary that $b_j \ge 0$ for all $j$.  Suppose further
that $b_i \ne 0$.  Then by letting
\begin{align*}
	q_{p+1}(u,t) = \begin{cases} 0 & t \ne t_0 + c_i \Delta t \\
    					   1 & t = t_0 + c_i \Delta t,
    \end{cases}
\end{align*}
we obtain
\begin{align*}
    u^{n+1}_{p+1} = b_i \frac{\Delta t}{\Delta x} (y^i_p - y^i_{p+1}) < 0.
\end{align*}
The last inequality follows by deducing from the construction above that $y^i_{p+1}=0$.

Finally, suppose that $b_i=0$, so that stage $i$ is not used directly
to compute the new solution.  Since the method is DJ-reducible, there
exists some sequence of indices $i_1, i_2,\dots,i_r=m+1$ such that
$a_{i,i_1}a_{i_1,i_2} \cdots b_r \ne 0$, and by a similar construction
we can ensure that each of the stages $y_{i_1}, y_{i_2}, \dots, y_{i_r} = u^{n+1}$
has a negative entry.
\end{proof}

\subsection{Upper bounds for the step-size coefficient}

The step-size coefficient $\gamma$---depending only on the chosen RK method---is a constant
that guarantees non-negativity of the RK recursion under the time-step restriction 
\eqref{timesteprestriction} for the whole class of problems \eqref{eq:ODEs}.
We can find upper bounds for it by considering more restricted problem sets.

First we consider the constant-coefficient linear advection equation
$U_t + U_x = 0$
semi-discretized with first-order upwind differences, leading to an
ODE system $u'(t) = Lu(t)$ that is a special case of \eqref{eq:ODEs} with $q=1$.
A Runge--Kutta method applied to this problem results in the iteration
$$u^{n+1} = \varphi_{A,b}(\Delta t L) u^n,$$ 
where $(A,b)$ are the coefficients of the Runge--Kutta method and 
$\varphi_{A,b}$ is the stability function of the method.
It can be shown
(see \cite[Theorem 4.2]{David11}\label{thm:linstepsize}; this result
is also a straightforward adaptation of the seminal result in \cite{bolley1978})
that the matrix $\varphi_{A,b}(\Delta t L)$ is non-negative if and only if
$\Delta t \le R(\varphi) \Delta x,$ where $R(\varphi)$ is the {\em radius
of absolute monotonicity} of $\varphi$ (also known as the {\em threshold factor};
see the citations above).
This leads to
\begin{prop}\label{corRphi}
    The step-size coefficient $\gamma{(A,b)}$ is no larger than $R(\varphi_{A,b})$.
\end{prop}

Our step-size coefficient is also upper-bounded by the {\em modified threshold factor}
$R_{\mathcal M}(A,b)$ described in \cite{H13}.  This can be seen
by considering an advection equation with time-dependent advection speed, which
leads to a semi-discretization that fits in the class of problems considered in
\cite{H13}.  This second bound is in general sharper, but its value
is known only for a small number of methods (see \cite[Table 2]{H13}).


\section{Step-size coefficients for second-order methods}\label{sect:ERK(2,2)}

In this section we summarize and comment on the results of \cite{K10} about comparing the optimal
SSP coefficient $\ssp$ and the optimal step-size coefficient
 $\gamma$ within the class of two-stage second-order ERK methods (ERK(2,2)).
It is known \cite[Section 32]{Butcher08} that all ERK(2,2) methods can be described by a one-parameter family with Butcher tableau 
\begin{equation}\label{ButcherERK22}
A = 
  \begin{pmatrix} 0 & 0 \\ \alpha & 0 \end{pmatrix},\quad
  b^\top=\begin{pmatrix} 1-\frac{1}{2\alpha},  \frac{1}{2\alpha} \end{pmatrix},
\end{equation}
where $\alpha\in \mathbb{R}\setminus\{0\}$. 

The SSP coefficient of an ERK(2,2) method with parameter 
$\alpha$ is determined in \cite{K10} as\footnote{Please note that the derivation of this formula in \cite[Section 2]{K10} (denoted by $\gamma(\kappa)$ there) contains some inconsistencies.}
\begin{equation}\label{Calphaexplicit}
\ssp_\alpha=
\begin{cases}
0,& 0\ne \alpha<\frac{1}{2}\\
2-\frac{1}{\alpha}, & \frac{1}{2}\leq\alpha\leq 1\\
\frac{1}{\alpha}, & \alpha>1.
\end{cases}
\end{equation}
This means that we have the maximum SSP coefficient $\ssp_\alpha=1$ within the
ERK(2,2) family if and only if $\alpha=1$.  The corresponding RK method is
known as the improved Euler's method, explicit trapezoidal rule, or Heun's
method.

Now we turn our attention to the optimal step-size coefficient. 
When a two-stage ERK method with (strictly lower-triangular) Butcher matrix $A=(a_{jk})$ and vector $b=(b_j)$
is applied to \eqref{eq:ODEs},  the $k^\textrm{th}$ component of the solution in \eqref{eq:ERKstep-alt} becomes
\[
u^{n+1}_k=P_0(\xi) u^{n}_k+P_1(\xi) u^{n}_{k-1}+P_2(\xi) u^{n}_{k-2}
\]
with
\begin{subequations} \label{P_2stage}
\begin{align} 
P_0(\xi) =\ & 1-b_1\xi_k^{1}-b_2\xi_k^{2}+a_{21}b_2\xi_k^{1}\xi_k^{2}, \nonumber\\
P_1(\xi) =\ & b_1\xi_k^{1}+b_2\xi_k^{2}-a_{21}b_2\xi_{k-1}^{1}\xi_k^{2}-a_{21}b_2\xi_k^{1}\xi_k^{2}, \label{eq:2stagecoeff}\\
P_2(\xi) =\ & a_{21}b_2\xi_{k-1}^{1}\xi_k^{2}.\nonumber
\end{align}
By using \eqref{ButcherERK22}, for ERK(2,2) methods we obtain
\begin{align} 
P_0\left(\xi_{k-1}^{1},\xi_k^{1},\xi_k^{2}\right) =\  & 1-\left(1-\frac{1}{2\alpha}\right)\xi_k^{1}-\frac{1}{2\alpha}\xi_k^{2}+\frac{1}{2}\xi_k^{1}\xi_k^{2}, \nonumber\\
P_1\left(\xi_{k-1}^{1},\xi_k^{1},\xi_k^{2}\right) =\  & \left(1-\frac{1}{2\alpha}\right)\xi_k^{1}+\frac{1}{2\alpha}\xi_k^{2}-\frac{1}{2}\xi_{k-1}^{1}\xi_k^{2}-\frac{1}{2}\xi_k^{1}\xi_k^{2}, \label{eq:2stagecoeffreduced}\\
P_2\left(\xi_{k-1}^{1},\xi_k^{1},\xi_k^{2}\right) =\  & \frac{1}{2}\xi_{k-1}^{1}\xi_k^{2}.\nonumber
\end{align}
\end{subequations}
The following proposition improves on \cite[Theorem 1]{K10} by establishing the sharpness of the step-size coefficient.
\begin{prop} \label{prop:ERK22}
    The positivity step-size coefficient for the family of methods 
    \eqref{ButcherERK22} is given by
\begin{equation}\label{gammainERK22}
    \gamma_\alpha=
    \begin{cases}
    0,& 0\ne \alpha<\frac{1}{2}\\
    1, & \frac{1}{2}\leq\alpha\leq 1\\
    \frac{1}{\alpha}, & \alpha>1.
    \end{cases}
\end{equation}
\end{prop}
\begin{proof}
Regarding the $\alpha<0$ case (not appearing explicitly in \cite[Theorem 1]{K10}), we observe
that 
\[
P_1\left(0,0,\varepsilon \right)=\frac{\varepsilon}{2 \alpha }<0\quad \textrm{for any } \alpha<0 \textrm{ and } \varepsilon>0,
\]
implying $\gamma_\alpha=0$ for $\alpha<0$.

For $0<\alpha<1/2$ we have $\gamma_\alpha=0$ because
\[
P_1\left(\varepsilon ,\varepsilon,0 \right)=\frac{(2 \alpha -1) \varepsilon }{2 \alpha }<0\quad \textrm{for any } \alpha\in \left(0, \frac{1}{2}\right) \textrm{ and } \varepsilon>0.
\]  

For $1/2\le\alpha\le 1$,  \cite[Theorem 1]{K10} shows that 
$P_i\ge 0$ ($i=0,1, 2$) for any $\xi \in [0,1]^3$, implying $\gamma_\alpha\le 1$ here. By noticing that
\[
2P_1(1,1+\varepsilon,1+\varepsilon)=-\varepsilon(1+\varepsilon)<0\quad \textrm{for any } \alpha\in [1/2,1] \textrm{ and } \varepsilon>0,
\]
we see that $\gamma_\alpha=1$ for $\alpha\in[1/2,1]$.

Finally, for $\alpha>1$,  \cite[Theorem 1]{K10} proves that 
$P_i(\xi)\ge 0$ ($i=0,1, 2$) for any $\xi \in [0,1/\alpha]^3$, implying $\gamma_\alpha\le 1/\alpha$. By
considering the inequality
\[
P_1\left(\frac{1}{\alpha }+\frac{3 \varepsilon
   }{4},\frac{\varepsilon  (\alpha  \varepsilon +2)}{2 (4 \alpha
   -4-\alpha \varepsilon)},\frac{1}{\alpha }+\frac{\varepsilon }{2}\right)=-\frac{\varepsilon  (\alpha  \varepsilon +2)}{16 \alpha }<0\quad \textrm{for any } \alpha>1  \textrm{ and } 
 \varepsilon \in \left(0,1-\frac{1}{\alpha}\right),
\]
and taking into account that here all three arguments of $P_1$ are located in the interval $\left(0,\frac{1}{\alpha}+\varepsilon\right)$, and that $\varepsilon>0$ can be chosen arbitrarily close to 0, we conclude that
$\gamma_\alpha= 1/\alpha$.
\end{proof}
\begin{rem}
Proposition \ref{prop:ERK22} provides a theoretical basis for observations in 
\cite{hundsdorfer1995positive}, where the methods \eqref{ButcherERK22} with
$\alpha=1$ and $\alpha=1/2$ were observed to behave identically with respect
to positivity.  Similar observations were made in \cite{2005_ketcheson_robinson}
regarding total variation.
\end{rem}

\begin{rem}\label{ERKcor2Remark} Proposition \ref{corRphi} directly yields that $\gamma_\alpha>1$ 
cannot hold for any $\alpha\ne 0$, since the radius of absolute monotonicity of the stability
function of any ERK(2,2) method is at most 1 (see \cite[Section 4.8]{David11}).
\end{rem}

As a conclusion, by comparing \eqref{Calphaexplicit} and \eqref{gammainERK22}, it is seen that the step-size coefficient is strictly larger than
the SSP coefficient for ERK(2,2) methods with $\alpha\in [1/2,1)$.

\section{Step-size coefficients for third-order methods}\label{sect:ERK(3,3)}

In this section we consider the class of three-stage third-order ERK methods (ERK(3,3)), and compare their
optimal SSP and step-size coefficients. The ERK(3,3) class is a disjoint union of three subclasses, referred to as
Cases I, II and III in \cite[Section 32]{Butcher08}. Case I is a two-parameter family of methods, while any of Cases II and III is a one-parameter family of methods.

It is known \cite[Section 2.4.2]{David11} that any ERK(3,3) 
method satisfies $\ssp_{\alpha,\beta}\le 1$,
and the optimal value $\ssp= 1$ is achieved only by the Case I method with
parameters $\alpha=1$ and $\beta=1/2$.

Regarding the optimum value of the step-size coefficient $\gamma$,  Proposition \ref{corRphi} shows that $\gamma\le 1$ should hold for any ERK(3,3) method, because the radius of absolute monotonicity of the stability function of any ERK(3,3) method is at most 1 (cf.~Remark \ref{ERKcor2Remark}). 

In the rest of this section we investigate whether $\gamma=1$ can be achieved in the ERK(3,3) family.  To this end, we first generate
the multivariable polynomials appearing in Definition \ref{defn1} (cf.~\eqref{eq:2stagecoeff}). On applying a three-stage ERK method to problem \eqref{eq:ODEs}, the
$k^\mathrm{th}$ component of the step solution in \eqref{eq:ERKstep-alt} is
$$u^{n+1}_k=P_0(\xi) u^{n}_k+P_1(\xi) u^{n}_{k-1}+P_2(\xi) u^{n}_{k-2}+P_3(\xi) u^{n}_{k-3},$$
where 
\begin{align} 
P_0(\xi)=&\ 1-b_1\xi_k^{1}-b_2\xi_k^{2}+a_{21}b_2\xi_k^{1}\xi_k^{2}-b_3\xi_{k}^{3}+a_{31}b_3\xi_{k}^{1}\xi_{k}^{3}+a_{32}b_3\xi_{k}^{2}\xi_{k}^{3}- \nonumber\\
&a_{32}a_{21}b_3\xi_{k}^{1}\xi_{k}^{2}\xi_{k}^{3}, \nonumber\\
P_1(\xi)=&\  b_1\xi_k^{1}+b_2\xi_k^{2}-a_{21}b_2\xi_{k-1}^{1}\xi_k^{2}-a_{21}b_2\xi_k^{1}\xi_k^{2}+b_3\xi_{k}^{3}-a_{31}b_3\xi_{k-1}^{1}\xi_{k}^{3}- \nonumber\\
&a_{32}b_3\xi_{k-1}^{2}\xi_{k}^{3}+a_{32}a_{21}b_3\xi_{k-1}^{1}\xi_{k-1}^{2}\xi_{k}^{3}-a_{31}b_3\xi_{k}^{1}\xi_{k}^{3}-a_{32}b_3\xi_{k}^{2}\xi_{k}^{3}+ \nonumber\\
&a_{32}a_{21}b_3\xi_{k-1}^{1}\xi_{k}^{2}\xi_{k}^{3}+a_{32}a_{21}b_3\xi_{k}^{1}\xi_{k}^{2}\xi_{k}^{3}, \label{P_3stage}\\
P_2(\xi)=&\  a_{21}b_2\xi_{k-1}^{1}\xi_k^{2}+a_{31}b_3\xi_{k-1}^{1}\xi_{k}^{3}+a_{32}b_3\xi_{k-1}^{2}\xi_{k}^{3}-a_{32}a_{21}b_3\xi_{k-2}^{1}\xi_{k-1}^{2}\xi_{k}^{3}- \nonumber\\
&a_{32}a_{21}b_3\xi_{k-1}^{1}\xi_{k-1}^{2}\xi_{k}^{3}-a_{32}a_{21}b_3\xi_{k-1}^{1}\xi_{k}^{2}\xi_{k}^{3}, \nonumber\\
P_3(\xi)= &\  a_{32}a_{21}b_3\xi_{k-2}^{1}\xi_{k-1}^{2}\xi_{k}^{3} \nonumber.
\end{align}
For readability we define
\begin{equation}\label{xyzuvw}
x:=\xi_{k-2}^{1},\ y:=\xi_{k-1}^{1},\ z:=\xi_k^{1},\ u:=\xi_{k-1}^{2},\ v:=\xi_k^{2},\ w:=\xi_k^{3}.
\end{equation}

\subsection{Case I}

In this subsection we focus on Case I, referred to as \textit{generic} ERK(3,3) methods. 
Their corresponding Butcher tableau with real parameters $\alpha$ and $\beta$ satisfying
\begin{equation}\label{ERK33alphabetarestriction}
\alpha, \beta\in \mathbb{R}\setminus\{0\} \quad \mathrm{and}\quad \frac{2}{3}\ne \alpha\neq\beta
\end{equation} 
is given by 
\begin{equation}\label{RK2parameterfamily}
A=\left(
\begin{array}{ccc}
 0 & 0 & 0 \\
 \alpha  & 0 & 0 \\
 \beta-\frac{(\alpha -\beta ) \beta }{\alpha  (3 \alpha -2)}  & \frac{(\alpha -\beta ) \beta }{\alpha  (3 \alpha -2)} & 0 \\
\end{array}
\right),\quad 
b^\top=\left(\frac{6 \alpha  \beta -3 \alpha -3 \beta +2}{6 \alpha  \beta },\frac{2-3 \beta }{6 \alpha  (\alpha -\beta )},\frac{3 \alpha -2}{6 \beta  (\alpha -\beta)}\right).
\end{equation}
By replacing $a_{ij}$ and $b_j$ in \eqref{P_3stage} with their corresponding parametrizations in \eqref{RK2parameterfamily}, 
we get
\begin{align*}
P_{0}(x,y,z,u,v,w)=\ &\frac{1}{6 \alpha  \beta  (\alpha -\beta )}\biggl[
-\alpha ^2 \beta  (v w z-3 w z+6 z-6)+\alpha  \beta ^2 (v w z-3 v z+6 z-6)+&\\
&\alpha  \beta 
   (v w+2 v z-3 w z)-\beta ^2 (w-3) (v-z)-2 \beta  (v-z)-3 \alpha ^2 (w-z)+2 \alpha 
   (w-z)\biggr],&\\
P_1(x,y,z,u,v,w)=\ &\frac{1}{6 \alpha  \beta  (\alpha -\beta )}\biggl[
\alpha ^2 \beta  (u w y+v w y+v w z-3 w y-3 w z+6 z)-&\\
&\alpha  \beta ^2 [u w y+v (w-3)
   (y+z)+6 z]-\alpha  \beta  (u w+v w+2 v y+2 v z-3 w y-3 w z)+&\\
&\beta ^2 [u w+v (w-3)-w
   y-w z+3 z]+2 \beta  (v-z)+3 \alpha ^2 (w-z)-2 \alpha  (w-z)\biggr],&\\
P_2(x,y,z,u,v,w)=\ &\frac{1}{6 \alpha  (\alpha -\beta )}\biggl[-\alpha ^2 w[u (x+y)+(v-3) y]+&\\
&\alpha  \beta  [u w (x+y)+v y (w-3)]+\alpha  (u w+2 v
   y-3 w y)+\beta  w (y-u)\bigg],&\\
P_3(x,y,z,u,v,w)=\ &\frac{1}{6}x u w.&
\end{align*}
To emphasize the dependence on the parameters, we will write $P_{i,\alpha,\beta}$ instead of
$P_{i}$. 

To characterize all methods in the generic ERK(3,3) family having the
maximum step-size coefficient $\gamma=1$,
we try and find (possibly all) pairs $(\alpha,\beta)\in\mathbb{R}^2$ satisfying \eqref{ERK33alphabetarestriction} such that for any $i=0,1, 2, 3$ we have
\begin{equation}\label{Pialphabetanonneg}
P_{i,\alpha,\beta}\ge 0 \quad \textrm{in } [0,1]^6.
\end{equation}
The polynomial $P_{3,\alpha,\beta}$ is clearly non-negative, so it is sufficient to deal with the indices
 $i=0,1, 2$.\\

\noindent \textbf{Remark.} \textit{Compared to Section \ref{sect:ERK(2,2)}, the computations are now much more involved. This explains why we will not attempt to compute the step-size coefficient $\gamma_{\alpha,\beta}$ for each generic ERK(3,3) method $(\alpha,\beta)$ (but cf. \eqref{gammaalphaforCaseII} and \eqref{gammaalphaforCaseIII}).}\\

First we formulate some necessary conditions on the parameters $(\alpha,\beta)$ for \eqref{Pialphabetanonneg} to hold. Observe that
\[
P_{2,\alpha,\beta}(0,0,1,1,1,1)=\frac{1}{6\alpha},
\]
implying $\alpha>0$. Then 
\[
P_{2,\alpha,\beta}(1,0,1,1,1,1)=\frac{1}{6} \left(\frac{1}{\alpha }-1\right)
\]
shows that $\alpha\le 1$ should also hold. Next we consider  
\[
P_{2,\alpha,\beta}(0,1,0,0,1,1)=\frac{2 \alpha -1}{6 \alpha },
\]
so $\alpha\ge 1/2$. Now we take into account that the non-negativity of 
\[
P_{2,\alpha,\beta}(1,1,1,1,1,0)=\frac{2-3 \beta }{6 (\alpha -\beta) } \quad \textrm{and}\quad
P_{2,\alpha,\beta}(1,1,0,1,0,1)=\frac{\alpha +2 \beta -2}{6 (\alpha -\beta )}
\]
together with \eqref{ERK33alphabetarestriction} and $\alpha\in[1/2,1]\setminus\{2/3\}$ imply that
\begin{equation}\label{bowtiedef}
(\alpha,\beta)\in {\cal{B}}:=\left\{ (\alpha,\beta)\in\mathbb{R}^2 : \left(\frac{1}{2}\leq \alpha <\frac{2}{3}\land \frac{2}{3}\leq \beta \leq 1-\frac{\alpha
   }{2}\right)\lor \left(\frac{2}{3}<\alpha \leq 1\land 1-\frac{\alpha }{2}\leq \beta
   \leq \frac{2}{3}\right)\right\}
\end{equation}
is necessary for \eqref{Pialphabetanonneg} to hold ($i=0,1,2$), see Figure \ref{bowtiefigure}. In fact, we have systematically evaluated the polynomials $P_{0,\alpha,\beta}$, $P_{1,\alpha,\beta}$ and $P_{2,\alpha,\beta}$ at each of the 64 vertices of the hypercube $[0,1]^6$ to choose the relevant polynomial and vertices presented above (see also Appendix \ref{appendixB}).

We now claim that \eqref{bowtiedef} is also sufficient for \eqref{Pialphabetanonneg}. 

\begin{conj}\label{conj1}
All methods \eqref{RK2parameterfamily} with $(\alpha,\beta)\in {\cal{B}}$ have $\gamma = 1$.
\end{conj}
The conjecture is based on the following computations. We sampled the parameter set ${\cal{B}}$ at the grid points ${\cal{G}}\subset{\cal{B}}$ shown in Figure \ref{bowtiefigure}, and
verified that \eqref{Pialphabetanonneg} holds for each $i=0,1,2$
and $(\alpha,\beta)\in {\cal{G}}$. The \textit{Mathematica} commands {\texttt{Reduce}}
and {\texttt{FindInstance}} were key to formulating the above conjecture.
Proving the full conjecture would require a significant amount of work.

\begin{figure}
\begin{center}
\includegraphics[width=0.8\textwidth]{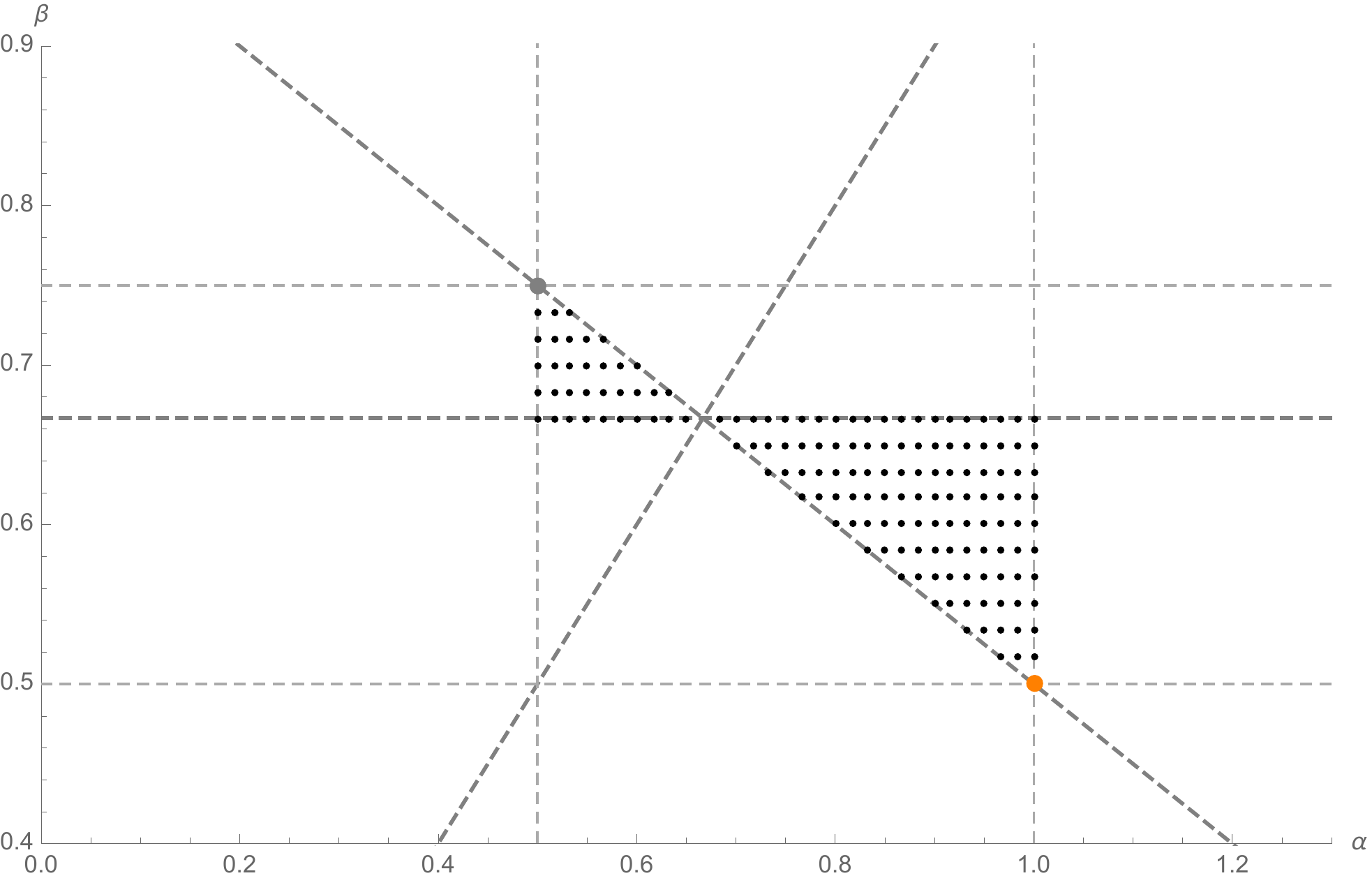}
\caption{The figure shows the bowtie-shaped region ${\cal{B}}$ (see \eqref{bowtiedef}) together
with the grid ${\cal{G}}\subset{\cal{B}}$ in the parameter plane. The line $\beta=\alpha$ is also
displayed (cf.~\eqref{ERK33alphabetarestriction}). 
The orange dot at $(\alpha,\beta)=(1,1/2)$ corresponds to the unique
optimal ERK(3,3) method with $\ssp=1$. The gray dot at $(\alpha,\beta)=(1/2,3/4)$ corresponds to the 
unique ERK(3,3) method
with minimum truncation-error coefficient.\label{bowtiefigure}}
\end{center}
\end{figure}

In particular, Conjecture \ref{conj1} means that  there are infinitely many
ERK(3,3) methods with step-size coefficient $\gamma=1$ in the present Case I (but see also Proposition \ref{prop5withproof}, where we are going to actually prove that there are infinitely many
ERK(3,3) methods with step-size coefficient $\gamma=1$ in Case II).  In contrast, as stated
earlier, there is only one ERK(3,3) method with SSP coefficient $\ssp=1$
(denoted by the orange dot in Figure \ref{bowtiefigure}). 


It is a nice coincidence that the method with $(\alpha,\beta)=(1/2, 3/4)$ (depicted as the gray vertex of ${\cal{B}}$ in Figure \ref{bowtiefigure}), that is, the one with tableau
\begin{equation}\label{methodwithmintrunc}
A=\left(
\begin{array}{ccc}
 0 & 0 & 0 \\
\frac{1}{2}  & 0 & 0 \\
0  & \frac{3}{4} & 0 \\
\end{array}
\right),\quad 
b^\top=\left(\frac{2}{9},\frac{1 }{3},\frac{4}{9}\right)
\end{equation}
is also an element of ${\cal{B}}$: in the full ERK(3,3) family (i.e., in Cases I-III) this is the 
unique method having the minimum truncation-error coefficient \cite[p.~433]{RA1962}.
Therefore, for problems of type \eqref{eq:ODEs}, the method \eqref{methodwithmintrunc} 
satisfies two optimality criteria simultaneously. The following proposition shows that this method indeed has the optimal step-size coefficient
$\gamma=1$.

\begin{prop} For each $i=0,1,2$ and all $(x,y,z,u,v,w)\in [0,1]^6$ we have
\[
P_{i,1/2,3/4}(x,y,z,u,v,w)\ge 0.
\]
\end{prop}
\begin{proof} For any $(x,y,z,u,v,w)\in [0,1]^6$ we have
\[
18 P_{0,1/2,3/4}(x,y,z,u,v,w)=18-8 w-4 z+3 v (1-w) (z-2)\ge 6+3 v (1-w) (z-2)\ge 6-6=0,
\]
\[
18 P_{1,1/2,3/4}(x,y,z,u,v,w)=3 v (1-w) (2-y-z)+4 z+w (8+3 u (y-2))\ge w (8+3 u (y-2))\ge 2w\ge 0,
\]
and
\[
6P_{2,1/2,3/4}(x,y,z,u,v,w)=u w (2-x-y)+y v (1-w)\ge 0.
\]
\end{proof}

\begin{rem}
The present approach does not fully explain results in \cite{hundsdorfer1995positive}
regarding 3rd-order methods.  Therein, the method with $\alpha=1/3, \beta=2/3$
was found to give good results, but investigation of this method using the
present technique yields that $\gamma=0$.
\end{rem}

\subsection{Case II}\label{sec:CaseII}

For any $\alpha\in \mathbb{R}\setminus\{0\}$, these ERK(3,3) methods are described by the tableau
\begin{equation}\label{RK2parameterfamilyCaseII}
A=\left(
\begin{array}{ccc}
 0 & 0 & 0 \\
 \frac{2}{3} & 0 & 0 \\
 \frac{2}{3}-\frac{1}{4 \alpha } & \frac{1}{4 \alpha } & 0 \\
\end{array}
\right),\quad 
b^\top=\left(\frac{1}{4},\frac{3}{4}-\alpha,\alpha \right).
\end{equation}
Regarding the SSP coefficient in this family, we have the following result (whose proof details are omitted here).
\begin{prop}
The SSP coefficient satisfies
\begin{equation}\label{sspforCaseII}
\sspa=
\begin{cases}
0,& 0\ne \alpha< \frac{3}{8}\\
\frac{8 \alpha -3}{2}, & \frac{3}{8}\le \alpha\le \frac{9}{16}\\
3-4\alpha, & \frac{9}{16}<\alpha\le \frac{3}{4}\\
0, & \alpha> \frac{3}{4},
\end{cases}
\end{equation}
hence the unique maximum of the SSP coefficient
occurs at $\ssp_{9/16}=3/4$.
\end{prop}
\begin{prop}\label{prop5withproof}
The step-size coefficient in this family is given by 
\begin{equation}\label{gammaalphaforCaseII}
\gamma_\alpha=
\begin{cases}
0,& 0\ne \alpha<\frac{3}{8}\\
2\alpha,& \frac{3}{8}\le \alpha<\frac{1}{2}\\
1, & \frac{1}{2}\le \alpha\le\frac{3}{4}\\
0, & \alpha>\frac{3}{4}.
\end{cases}
\end{equation}
\end{prop}
\begin{proof}
We use the polynomials and variables given by \eqref{P_3stage}-\eqref{xyzuvw} now with 
\eqref{RK2parameterfamilyCaseII}. 

For any fixed $0\ne \alpha<\frac{3}{8}$ and arbitrary 
$\varepsilon>0$ we have
\[
12 P_{2,\alpha}(0,\varepsilon,0, 0, 0, \varepsilon)=\varepsilon^2 (8\alpha-3)<0,
\]
showing that $\gamma_\alpha=0$ here. 

For any fixed $\alpha>\frac{3}{4}$ and arbitrary 
$\varepsilon>0$ we have
\[
6 P_{2,\alpha}(\varepsilon, \varepsilon,  \varepsilon,  \varepsilon, \varepsilon, 0)=\varepsilon^2 (3-4\alpha)<0,
\]
hence $\gamma_\alpha=0$ also holds for these values of $\alpha$.

For any fixed $\frac{3}{8}\le \alpha<\frac{1}{2}$ and arbitrary 
$\varepsilon>0$ we have
\[
3 P_{2,\alpha}(2\alpha+\varepsilon, 2\alpha+\varepsilon,  2\alpha+\varepsilon,  2\alpha+\varepsilon,
 0, 2\alpha+\varepsilon)=-\varepsilon (2\alpha+\varepsilon)^2<0,
\]
so $\gamma_\alpha\le 2\alpha$ for these values of $\alpha$.

For any fixed $\frac{1}{2}\le \alpha\le \frac{3}{4}$ and arbitrary 
$\varepsilon>0$ we have
\[
2 P_{2,\alpha}(1+\varepsilon, 1+\varepsilon,  1+\varepsilon,  1+\varepsilon,
 1+\varepsilon, 1+\varepsilon)=-\varepsilon (1+\varepsilon)^2<0,
\]
therefore $\gamma_\alpha\le 1$ for these values of $\alpha$.

Since $P_3(x,y,z,u,v,w)=\frac{1}{6}x u w\ge 0$ for any non-negative choice of the arguments, to finish the proof of the proposition, we need to establish that 
\[P_{0,\alpha}(x,y,z,u,v,w)\ge 0, \quad P_{1,\alpha}(x,y,z,u,v,w)\ge 0\quad \text{and}\quad P_{2,\alpha}(x,y,z,u,v,w)\ge 0\] 
hold
\begin{itemize}
\item for any $\alpha\in\left[ \frac{3}{8}, \frac{1}{2}\right)$ and $(x,y,z,u,v,w)\in [0,2\alpha]^6$, implying $\gamma_\alpha\ge 2\alpha$;
\item for any $\alpha\in\left[ \frac{1}{2}, \frac{3}{4}\right]$ and $(x,y,z,u,v,w)\in [0,1]^6$, implying $\gamma_\alpha\ge 1$.
\end{itemize}
Here we present the proof only for the case $\alpha\in\left[ \frac{1}{2}, \frac{3}{4}\right]$---yielding the maximum possible value of $\gamma_\alpha=1$ for infinitely many ERK(3,3) methods; the proof details in the other case (for $\alpha\in\left[ \frac{3}{8}, \frac{1}{2}\right)$) are analogous.

$\bullet$ \textbf{Non-negativity of $P_{0,\alpha}$.} We first show that $P_{0,\alpha}(x,y,z,u,v,w)\ge 0$ holds for any $\alpha\in\left[ \frac{1}{2}, \frac{3}{4}\right]$ and $(x,y,z,u,v,w)\in [0,1]^6$. Indeed,
\[
12 P_{0,\alpha}(x,y,z,u,v,w)=12 \alpha  v-2 v w z+3 v w-8 \alpha  v z+6 v z-9 v-12 \alpha  w+8 \alpha  w z-
3 w z-3z+12=
\]
\[
(12 \alpha -9) v+v w (3-2 z)+2 (3-4 \alpha ) v z-12 \alpha  w+8 \alpha  w z-3(w+1) z+12\ge
\]
\[
\left(12\cdot\frac{1}{2}-9\right)\cdot 1+0+0-12 \alpha  w+8 \alpha  w z-3(1+1)z+12=-12 \alpha  w+8 \alpha  w z-6 z+9=
\]
\[
(3-2 z) (3-4 \alpha  w)\ge 0.
\]

$\bullet$ \textbf{Non-negativity of $P_{1,\alpha}$.} Next we show that $P_{1,\alpha}(x,y,z,u,v,w)\ge 0$ holds for any $\alpha\in\left[ \frac{1}{2}, \frac{3}{4}\right]$ and $(x,y,z,u,v,w)\in [0,1]^6$. This time we have
\[
12 P_{1,\alpha}(x,y,z,u,v,w)=
\big[-12 \alpha  v+8\alpha  v y-6 v y+8 \alpha  v z-6 v z+9 v+3 z\big]+
\]
\[
w \big[12 \alpha +2 u y-3 u+2 v y+2 v z-3 v-8 \alpha  y+3 y-8 \alpha  z+3 z\big].
\]
We prove that both pairs of brackets $[\ldots]$ above contain non-negative quantities.

As for the first pair of brackets, notice that
\[
-12 \alpha  v+8\alpha  v y-6 v y+8 \alpha  v z-6 v z+9 v+3 z=-12 \alpha  v+8 \alpha  v y-6 v y+8 \alpha  v z-3 v z+3 (1-v) z+9 v\ge
\]
\[
-12 \alpha  v+8 \alpha  v y-6 v y+8 \alpha  v z-3 v z+0+9 v=v (-12 \alpha +(8 \alpha -6) y+(8 \alpha-3) z+9)\ge
\]
\[
v (-12 \alpha +(8 \alpha -6) y+0+9)=v(3-4 \alpha ) (3-2 y)\ge 0.
\]

As for the second pair of brackets, we have
\[
12 \alpha +2 u y-3 u+2 v y+2 v z-3 v-8 \alpha  y+3 y-8 \alpha  z+3 z\ge
\]
\[
 12 \alpha +0\cdot u y-3\cdot 1+2 v y+2 v z-3 v-8 \alpha  y+3 y-8 \alpha  z+3 z.
\]
Now, by introducing $\rho:=2\alpha-1\in\left[0,\frac{1}{2}\right]$ and $\sigma:=y+z\in[0,2]=\left[0,\frac{3}{2}\right)\cup \left[\frac{3}{2},2\right]$, we get that the last expression above is equal to
\[
3-\sigma +(2 \sigma -3) (v-2 \rho ).
\]
For $\sigma\in\left[0,\frac{3}{2}\right)$ we are done because
\[
3-\sigma + (2 \sigma -3)(v-2 \rho )=3-\sigma + (2 \rho -v)(3-2 \sigma )\ge 3- \sigma+(0 -1)(3-2 \sigma )
   =\sigma \ge 0.
\]
For $\sigma\in\left[\frac{3}{2},2\right]$ we are also done because
\[
3-\sigma +(2 \sigma -3) (v-2 \rho )\ge 3-\sigma +(2 \sigma -3) (v-1)\ge 
\]
\[
1+(2 \sigma-3) (v-1)\ge 1+(2\cdot 2 -3)(v-1)=v\ge 0.
\]

$\bullet$ \textbf{Non-negativity of $P_{2,\alpha}$.} Finally we show that $P_{2,\alpha}(x,y,z,u,v,w)\ge 0$ holds for any $\alpha\in\left[ \frac{1}{2}, \frac{3}{4}\right]$ and $(x,y,z,u,v,w)\in [0,1]^6$. We now have
\[
12 P_{2,\alpha}(x,y,z,u,v,w)= (3-2  x)u w-2 u w y-2 v w y-8 \alpha  v y+6 v y+8 \alpha  w y-3 w y\ge
\]
\[
1\cdot  u w-2 u w y-2 v w y-8 \alpha  v y+6 v y+8 \alpha  w y-3 w y=\]
 \[
u w(1-y) -u w y-2 v w y-8\alpha  v y+6 v y+8 \alpha  w y-3 w y\ge
   \]
   \[
  0  -u w y-2 v w y-8 \alpha  v y+6 v y+8 \alpha  w y-3 w y=
   \]
   \[
   y(-u w-2 v w-8 \alpha  v+6 v+8 \alpha  w-3 w)\ge 
   y(-1\cdot w-2v w-8 \alpha  v+6 v+8 \alpha  w-3w)=
   \]
   \[
   y(-8 \alpha  v-2 v w+6 v+8 \alpha  w-4 w).
\]
Clearly, to finish the proof, it is enough to show that
\[
-8 \alpha  v-2 v w+6 v+8 \alpha  w-4 w\ge 0.
\]
By introducing $\rho:=2\alpha-1\in\left[0,\frac{1}{2}\right]$, the left-hand side of this last expression above becomes
\[
-4 (\rho +1) v-2 v w+6 v+4 (\rho +1) w-4 w.
\]
If $w-v\ge0$, then we are done, since
\[
-4 (\rho +1) v-2 v w+6 v+4 (\rho +1) w-4 w=4 \rho  (w-v)+2 v (1-w)\ge 0.
\]
If $w-v<0$, then we are also done, since
\[
-4 (\rho +1) v-2 v w+6 v+4 (\rho +1) w-4 w=4 \rho  (w-v)+2 v (1-w)\ge
\]
\[
4\cdot\frac{1}{2}\cdot (w-v)+ 2 v (1-w)=2 (1-v) w\ge 0.
\]
\end{proof}

Figure \ref{CvsGamma_CaseII} displays the coefficients $\sspa$ and $\gamma_\alpha$.

\begin{figure}
\begin{center}
\includegraphics[width=0.5\textwidth]{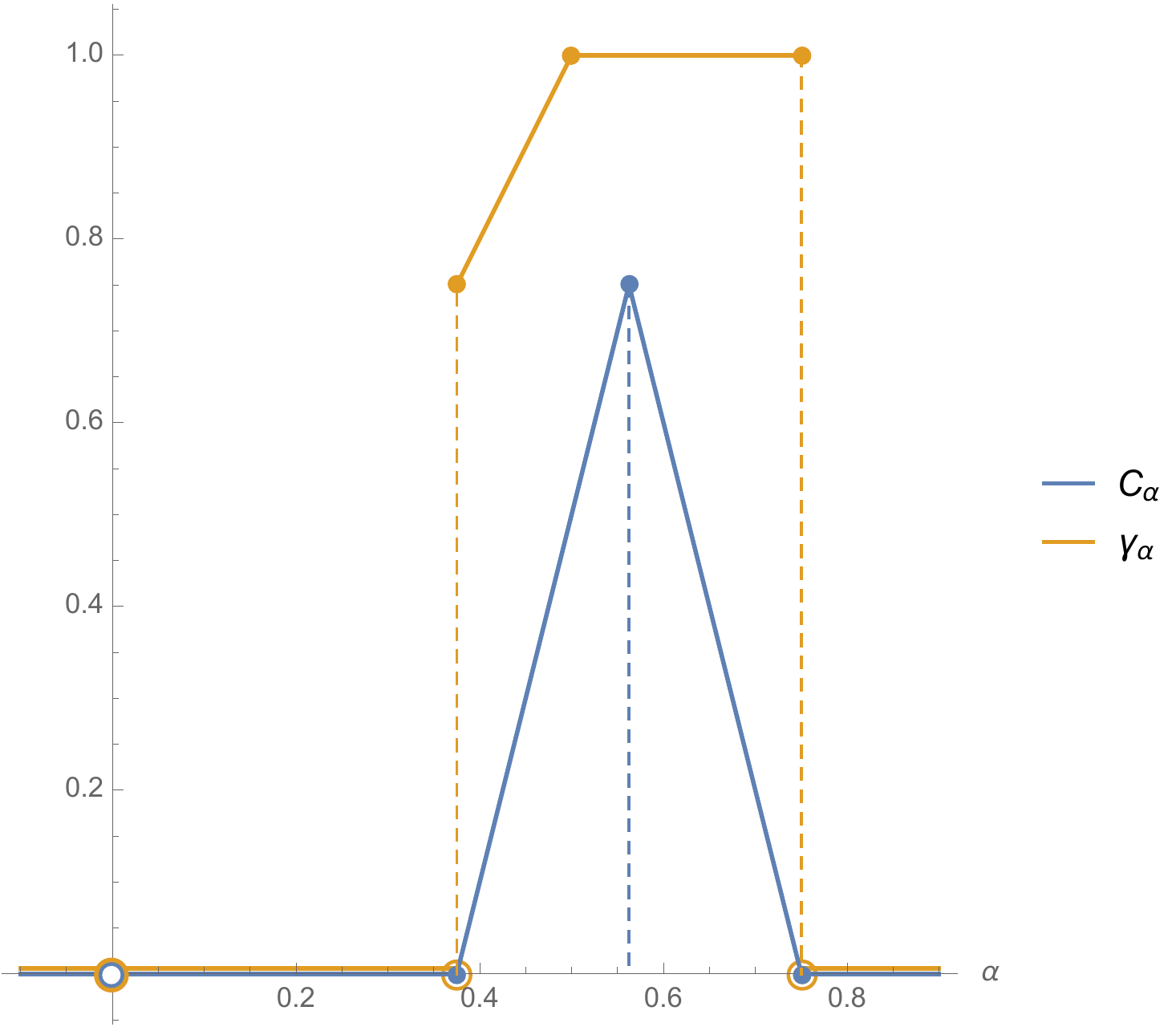}
\caption{The SSP coefficient \eqref{sspforCaseII} and the step-size coefficient 
\eqref{gammaalphaforCaseII} in Case II\label{CvsGamma_CaseII}}
\end{center}
\end{figure}

\subsection{Case III}\label{sec:CaseIII}

For any $\alpha\in \mathbb{R}\setminus\{0\}$, the ERK(3,3) methods within this family have the tableau
\begin{equation}\label{RK2parameterfamilyCaseIII}
A=\left(
\begin{array}{ccc}
 0 & 0 & 0 \\
 \frac{2}{3}  & 0 & 0 \\
  -\frac{1}{4\alpha}  & \frac{1}{4\alpha}  & 0 \\
\end{array}
\right),\quad 
b^\top=\left(\frac{1}{4}-\alpha,\frac{3}{4},\alpha \right).
\end{equation}
It can be proved that their SSP coefficient is trivial, that is,
\begin{equation}\label{sspforCaseIII}
\sspa=0 \textrm{ for any } \alpha\ne 0. 
\end{equation}
As for the step-size coefficient, by using the 
 polynomials and variables given in \eqref{P_3stage}-\eqref{xyzuvw} with \eqref{RK2parameterfamilyCaseIII}, we get for any $\varepsilon>0$ that
\[
4 P_{2,\alpha}(0,\varepsilon,0,0,0,\varepsilon)=-\varepsilon^2<0,
\]
implying 
\begin{equation}\label{gammaalphaforCaseIII}
\gamma_\alpha=0 \textrm{ for any } \alpha\ne 0.
\end{equation}

\section{Step-size coefficients for higher-order methods}\label{sect:ERKhigher}

In this section we investigate the non-negativity of higher-order ERK methods and show some negative results.

It is known \cite[Section 2.4.3]{David11} that the SSP coefficient for any ERK(4,4) method is $\ssp=0$.
The question naturally arises whether it is possible to find an ERK(4,4) method with positive
step-size coefficient.

First we recall the following result; see the proof of \cite[Theorem 9.6]{K91},\footnote{We thank Zolt\'an Horv\'ath (Sz\'echenyi Istv\'an University, Hungary) for pointing this out.} and of \cite[Theorem 2.4]{David11}.
\begin{prop}\label{nonnegERK44uniqueness}
A 4-stage 4th-order explicit Runge--Kutta method has non-negative Butcher tableau $(A, b)$ if and only if the method is the classical RK4 method (appearing in Appendix \ref{appendixA}).
\end{prop}

Proposition \ref{nonnegERK44uniqueness} and Theorem \ref{nonconfluenttheorem} yield the following.
\begin{prop}
Any non-confluent ERK(4,4) method has step-size coefficient $\gamma=0$.
\end{prop}

The previous result cannot be applied to the classical ERK(4,4) method, for example, which is confluent. For this method we have the following construction. 
\begin{prop}\footnote{Our Proposition 
    \ref{classicalERK44negativity} seems to directly contradict Theorem 1 in \cite{K15}.  To explain the discrepancy, note that the polynomial $P_3$ in our proof becomes negative along a 9-dimensional hyperface of the hypercube $[0,\varepsilon]^{10}$ for any 
 $\varepsilon>0$; in \cite{K15} it seems that the non-negativity of the corresponding (but slightly different) polynomial was checked only at the vertices of the hypercube $[0,1]^{10}$.}\label{classicalERK44negativity}
    There is no positive step-size coefficient $\gamma$ such that the
    classical ERK(4,4) method preserves positivity for all problems of the form
    \eqref{eq:ODEs}.
\end{prop}
\begin{proof}
    We prove the proposition by constructing an ODE system \eqref{eq:ODEs} such that the
    method gives a negative solution value for any positive step size.

First we determine (for example, by using the code in Appendix \ref{appendixA}) that the polynomial $P_3$ 
for this method appearing in \eqref{eq:poly}  is 
\[
P_3(x_1,x_2,x_3,x_4,x_5,x_6,x_7,x_8,x_9,x_{10})=
\]
\[
\frac{1}{24}\left(2x_2 x_6 x_9+2 x_5 x_8 x_{10}-x_1 x_5 x_8 x_{10}- x_2 x_5 x_8 x_{10}- x_2 x_6 x_8 x_{10}-x_2 x_6 x_{9} x_{10}\right),
\]
where, for better readability, we have relabeled the variables according to
\[
(x_1, x_2, \ldots, x_{10}):=(\xi^1_{k-3}, \xi^1_{k-2}, \xi^1_{k-1}, \xi^1_{k},
\xi^2_{k-2}, \xi^2_{k-1}, \xi^2_{k},\xi^3_{k-1}, \xi^3_{k},\xi^4_{k}).
\]
We observe that this polynomial is negative in any hypercube $[0,\varepsilon]^{10}$ with 
$\varepsilon>0$, since 
\[
P_3(0,\varepsilon,0,0,0,\varepsilon,0,\varepsilon,0,\varepsilon)=-\varepsilon^4/24<0.
\]
  
  We now let $N=4$, $\varepsilon = \Delta t/\Delta x$, and as initial data take $v := u(t_0)=(1,0,0,0)$.
  Then  $P_3(\xi) = -\varepsilon^4/24$ if
    $\xi_{k-2}^1 = \xi_{k-1}^2 = \xi_{k-1}^3 = \xi_k^4 = \varepsilon$
    and the remaining values of $\xi$ are zero.  Thus it remains only
    to choose a definition of the functions $q_k$ such that the
    aforementioned $\xi$ values result.

    We can write the first step of the method as
    \begin{align*}
        y^1   & = v \\
        y^2_k & = v_k + \frac{\varepsilon}{2} q_k(y^1) (y^1_{k-1}-y^1_k) \\
        y^3_k & = v_k + \frac{\varepsilon}{2} q_k(y^2) (y^2_{k-1}-y^2_k) \\
        y^4_k & = v_k + \varepsilon q_k(y^3) (y^3_{k-1}-y^3_k) \\
        u^1_k & = v_k + \frac{\varepsilon}{6}\left( 
                                     q_k(y^1)(y^1_{k-1}-y^1_k)
                                  + 2q_k(y^2)(y^2_{k-1}-y^2_k)
                                  + 2q_k(y^3)(y^3_{k-1}-y^3_k)
                                  +  q_k(y^4)(y^4_{k-1}-y^4_k)
                                        \right).
    \end{align*}
We have $y^1 = u(t_0) = (1,0,0,0)$, and we define $q_2(y^1)=1$, while
    $q_1(y^1)=q_3(y^1)=q_4(y^1)=0$.  
    This leads to $y^2=(1,\varepsilon/2,0,0)$, and we define $q_3(y^2)=1$, while
    $q_1(y^2)=q_2(y^2)=q_4(y^2)=0$.
    This leads to $y^3=(1,0,\varepsilon^2/4,0)$, and we define $q_3(y^3)=1$, while
    $q_1(y^3)=q_2(y^3)=q_4(y^3)=0$.
    This leads to $y^4=(1,0,-\varepsilon^3/4,0)$, and we define $q_4(y^4)=1$,
    while $q_1(y^4)=q_2(y^4)=q_3(y^4)=0$.
    This leads to $u^1 = (1, \varepsilon/6, (2\varepsilon^2-\varepsilon^3)/12, -\varepsilon^4/24)$.
    The last entry is negative for any positive step size, and the proof is complete.
\end{proof}

Some well-known ERK methods of higher order or with more stages---e.g., the
methods of Fehlberg \cite{fehlberg1969klassische} or 
Dormand-Prince \cite{Dormand1980}---contain negative entries in their Butcher tableau, so
Theorem \ref{nonconfluenttheorem} shows that their step-size coefficient is 0. 
Nevertheless, we have constructed some ERK(5,7) methods with non-negative tableau (see \cite{RS02} also).  For these methods however the approach described in Section \ref{sect:suffcond} becomes practically unmanageable with our current tools: one would need to test the non-negativity of multivariable polynomials with several hundred terms in high-dimensional hypercubes.

\section{Discussion and further applications}\label{sec:discussion}
The method described and applied herein can be employed to study positivity
and related properties for any semi-discretization that can be written in the
form
\begin{align*}
    u'_k(t) & = q_k(u(t),t) \sum_{j=-r}^r c_j u_{k-j}(t),
\end{align*}
where $q_k \ge 0$.
Given this form, we can determine polynomials such that
\begin{align} \label{eq:polysym}
    u_{k}^{n+1} & = \sum_{i=-rm}^{rm} P_i(\xi) u_{k-i}^n,
\end{align}
following the approach in Section \ref{sec:polycomp};
only the structure of the matrix $D$ changes in
\eqref{eq:Y}.  

The computationally expensive step in the analysis is the
determination of a hypercube---in the non-negative orthant, and preferably having the maximum edge length---in which the polynomials $P_i$ (depending on parameters if we are to optimize in parametric families
of ERK methods) are simultaneously non-negative.
The limiting factor is the dimension of the hypercubes, as it grows quadratically with the number
of stages of the method.  With current tools, methods with more than
five stages cannot easily be studied in this manner.  
Nevertheless, the necessary condition for non-negativity obtained by evaluating the polynomials $P_i$ at the hypercube vertices often turns out to be sufficient as well (sufficiency can be proved by 
examining the critical points and the boundaries).
Moreover, Theorem
\ref{nonconfluenttheorem} immediately provides negative conclusions regarding
many high-order methods.

It is worthwhile to return to a comparison of this approach with that of
strong stability preserving methods.  
The essential differences are the following.
\begin{itemize}
    \item In the current approach, we consider a more restricted class of
        problems; however, this class contains the principal class of
        problems for which SSP theory was developed.
    \item In the current approach, it is not necessary that the forward
        Euler method be positive in order to prove positivity for higher-order
        methods.  Nevertheless, in all cases studied so far, we have obtained
        positive step sizes only for semi-discretizations that are stable
        under forward Euler integration.
    \item The current approach gives larger step sizes than the SSP
        approach for many methods.
\end{itemize}
Both approaches are still overly pessimistic for certain methods of interest, such
as the classical 4-stage 4th-order method, which cannot be guaranteed positive
with either approach but yields good results in practice.  It may be possible
to obtain sharper step-size restrictions for such methods by either restricting
the allowed initial data or invoking some assumption of consistency on values
of $q$ generated from one stage to the next (both approaches would lead to
consideration of positivity of the polynomials $P_i$ on sets other than hypercubes).

In the remainder of this section, we demonstrate the application of these techniques
 to some additional semi-discretizations.

\subsection{Centered discretizations of hyperbolic problems}\label{sec:advPDEcentered}
Let us consider centered semi-discretizations of the form
\begin{align} \label{eq:skewsym}
    u'(t) = Q(u(t),t) D u(t),
\end{align}
where $Q$ is diagonal and $D$ is skew-symmetric.  Such discretizations arise
when centered finite differences are applied to a hyperbolic problem.
For example,
\begin{align}\label{eq:centeredODEs}
    u'_k(t) & = q_k(u(t),t) \frac{u_{k-1}(t)-u_{k+1}(t)}{2\Delta x},\quad k=1,\ldots,N
\end{align}
with $q=1$ is a common semi-discretization of the advection equation 
$U_t+U_x=0$.  Since the exact solution of this system of ODEs does not
preserve positivity, any consistent semi-discretization will also not be
positivity-preserving for sufficiently small step sizes.  Below, we show
how this conclusion can be reached directly using the present technique.

For such discretizations, we have the following result, which follows from the
fact that odd powers of $D$ are skew-symmetric while even powers are symmetric.
\begin{lem}
    Let a semi-discretization \eqref{eq:skewsym} be given with $Q$ diagonal and
    $D$ skew symmetric.  Define the vector $\hat{\xi}$ such that 
    $\hat{\xi}_{k+i}^j := \xi_{k-i}^j$.  For any Runge--Kutta method applied to
    this semi-discretization, we have an iteration of the form
    \eqref{eq:polysym}, where
    \begin{align*}
        P_j(\xi) & = P_{-j}(\hat{\xi}) & \text{for $j$ even;} \\
        P_j(\xi) & =-P_{-j}(\hat{\xi}) & \text{for $j$ odd.}
    \end{align*}
\end{lem}
The latter equality for the odd-numbered polynomials generally implies
that no positive step-size coefficient exists.


\subsection{The heat equation}\label{sec:heatPDE}

As a next example, we consider the scalar heat equation
$U_t=\kappa(x) U_{xx}$
and its semi-discretization in the form
\begin{align}\label{eq:heatODEs}
u'_k(t)=\displaystyle q_k\ \frac{u_{k-1}(t)-2u_k(t)+u_{k+1}(t)}{(\Delta x)^2},\quad k=1,\ldots,N,
\end{align}
where $u_0:=u_N,\ u_{N+1}:=u_1$, and $q_k = \kappa(x_k)\ge 0$.
The matrix $D_N$ in \eqref{eq:ERKstep-alt} is now the $N\times N$ tridiagonal matrix with entries $(1,-2,1)$ such that $[D_N]_{1,N}=[D_N]_{N,1}=1$. 

When, for example, a two-stage ERK method is applied to \eqref{eq:heatODEs}, the $k^\textrm{th}$ component of the step solution in \eqref{eq:ERKstep-alt} becomes
$$u^{n+1}_k=P_{-2}(\xi) u^{n}_{k+2}+P_{-1}(\xi) u^{n}_{k+1}+P_0(\xi) u^{n}_{k}+P_1(\xi) u^{n}_{k-1}+P_2(\xi) u^{n}_{k-2},$$
where
\begin{align} 
    P_{-2}(\xi)=\ & a_{21}b_2\xi_{k+1}^{1}\xi_k^{2}, \nonumber\\
    P_{-1}(\xi)=\ & b_1\xi_k^{1}+b_2\xi_k^{2}-2a_{21}b_2\xi_{k}^{1}\xi_k^{2}-2a_{21}b_2\xi_{k+1}^{1}\xi_k^{2}, \nonumber\\
    P_0(\xi)=\ & 1-2b_1\xi_k^{1}-2b_2\xi_k^{2}+a_{21}b_2\xi_{k-1}^{1}\xi_k^{2}+4a_{21}b_2\xi_{k}^{1}\xi_k^{2}+a_{21}b_2\xi_{k+1}^{1}\xi_k^{2}, \nonumber\\
    P_1(\xi)=\ & b_1\xi_k^{1}+b_2\xi_k^{2}-2a_{21}b_2\xi_{k-1}^{1}\xi_k^{2}-2a_{21}b_2\xi_{k}^{1}\xi_k^{2},\nonumber\\
    P_2(\xi)=\ & a_{21}b_2\xi_{k-1}^{1}\xi_k^{2}\nonumber
\end{align}
and
$$
    \xi_k^j = \frac{\Delta t}{(\Delta x)^2} q_k, \quad  j  = 1,\ldots,m.
$$
By introducing the variables 
\begin{align*}
x:=\xi_{k-1}^{1},\ y:=\xi_{k}^{1},\ z:=\xi_{k+1}^{1},\ u:=\xi_{k}^{2}
\end{align*}
and using \eqref{ButcherERK22}, the multivariable polynomials corresponding to the ERK(2,2) family are
now 
\begin{align*}
    P_{-2}(x,y,z,u)=\ & \frac{u z}{2},\\
    P_{-1}(x,y,z,u)=\ & \frac{u-y-2 \alpha  (y u+u z-y)}{2 \alpha},\\
    P_0(x,y,z,u)=\ & \frac{u (\alpha  x+4 \alpha  y+\alpha  z-2)+2 (\alpha -2 \alpha  y+y)}{2 \alpha },\\
    P_1(x,y,z,u)=\ & \frac{u-y-2 \alpha  (x u+y u-y)}{2 \alpha},\\
    P_2(x,y,z,u)=\ & \frac{x u}{2}.
\end{align*}
This time, unlike in Section \ref{sec:advPDEcentered}, we get non-trivial step-size coefficients
in the ERK(2,2) family.
\begin{prop} 
We have
 \[
    \gamma_\alpha
    \begin{cases}
    =1/2,& \alpha\in [1/2,1]\\
    <1/2, & 0\ne \alpha <1/2\textrm{ or } \alpha>1.
    \end{cases}
\]
\end{prop}
\begin{proof}
For any fixed $\alpha\ne 0$, we are looking for some $\varepsilon\equiv
\varepsilon(\alpha)>0$ such that the inequalities $P_i\ge 0$  for $-2\le i \le 2$
hold in $[0,\varepsilon]^4$.  Clearly, it is enough to consider the indices
$i=-1,0,1$. By again investigating the
vertices of the hypercube $[0,\varepsilon]^4$ (see Appendix \ref{appendixB}), we derive the necessary conditions
\[
\left( \frac{1}{2}\leq \alpha \leq 1 \text{ and } 0<\varepsilon \leq \frac{1}{2}\right)
\text{ or }
\left(\alpha>1 \text{ and } 0<\varepsilon \leq \frac{1}{2\alpha}\right).
\]
By considering the maximum value $\varepsilon=1/2$, we can then prove that 
for any $\alpha\in[1/2,1]$ and any $(x,y,z,u)\in [0,1/2]^4$ we indeed have 
$P_{-1}\ge 0$, $P_{0}\ge 0$ and $P_{1}\ge 0$.
\end{proof}


\appendix

\section{Some \textit{Mathematica} code}\label{sec:appendix}

\subsection{Generating the multivariable polynomials}\label{appendixA}

The first cell below contains the definition of a \textit{Mathematica} function {\texttt{ERKpolynomials}} for generating the multivariable polynomials in \eqref{eq:poly}. The two arguments $A$ and $b$ correspond to the Butcher tableau of the ERK method, and the output is a list of the $m+1$ polynomials $P_0$, \ldots, $P_m$ in the variables $\xi_\ell^j$. Note that 
the superscripts in $\xi_\ell^j$ are not exponents; the symbols $\xi_\ell^j$ with different sub- or superscripts denote different variables.

The second cell illustrates how to  
obtain the $4+1=5$ polynomials in $\frac{4\cdot (4+1)}{2}=10$ variables corresponding to the classical ERK(4,4) method.

\noindent \rule{\textwidth}{.4pt}
\begin{verbatim}
ERKpolynomials[A_,b_]:= Module[{m=Length@b,xi,S},ClearAll[\[Xi],k]; 
   xi=DiagonalMatrix@Table[Subsuperscript[\[Xi],k,i],{i,m}]; 
   CoefficientList[1+First[
   b.xi.Total@NestList[A.(DiscreteShift[S xi.#,{k,-1}]-xi.#)&,IdentityMatrix[m],m-1].
   ConstantArray[S-1,{m,1}]],S]]  
\end{verbatim}  
\rule{\textwidth}{.4pt}
\begin{verbatim}
ERKpolynomials[{{0,0,0,0},{1/2,0,0,0},{0,1/2,0,0},{0,0,1,0}},{1/6,1/3,1/3,1/6}] 
\end{verbatim}  
\rule{\textwidth}{.4pt}

\subsection{Non-negativity of polynomials at the vertices of a hypercube}\label{appendixB}

Here we provide a simple \textit{Mathematica} code to test the non-negativity of a multivariable polynomial
    by evaluating it at each vertex of a hypercube. In this particular example, the polynomial $P_{-1}(x,y,z,u)$
from Section \ref{sec:heatPDE} is evaluated at the $2^4$ vertices of the hypercube $[0,\varepsilon]^4$ with some $\varepsilon>0$, and the resulting system of 16 inequalities $P_{-1}\ge 0$ is solved.

\noindent \rule{\textwidth}{.4pt}
\quad $\texttt{Reduce\Big[}\varepsilon>0\texttt{\,\&\&\,And\,@@}\Big({\displaystyle 0\leq \frac{u-y-2 \alpha  (y\ u+u\  z-y)}{2 \alpha}}\texttt{/.}$\\ 
\\
\indent \quad \quad $\texttt{Thread\,/@\,Table\big[}{\displaystyle \{x,y,z,u\}\to}\texttt{\,Tuples[}\{0,\varepsilon\}\texttt{,4][[k]],}\{\texttt{k},1,2^4\}\texttt{\big]}\Big)\texttt{\Big]}$\\
\rule{\textwidth}{.4pt}
\bigskip
\bigskip
\noindent\textbf{Acknowledgement.} {We are indebted to the referees of the manuscript for their suggestions
that helped us improving the presentation of the material.}

\bibliographystyle{plain}
\bibliography{positivity.bib}

\end{document}